\documentclass{amsart}
\usepackage{amssymb}

\usepackage[pagebackref]{hyperref}

\usepackage{enumerate}
\usepackage[all]{xy}
\SelectTips{eu}{}
\CompileMatrices

\newcommand{\sst}{\scriptstyle}

\newcommand{\ges}{{\sst\geqslant}}
\newcommand{\les}{{\sst\leqslant}}

\newcommand{\betti}[3]{\beta_{#1}^{#2}(#3)}

\newcommand{\chr}{\operatorname{char}}

\newcommand{\depth}{\operatorname{depth}}

\newcommand{\hilb}[2]{\operatorname{Hilb}_{#1}(#2)}
\newcommand{\po}[3][S]{\operatorname{P}^{#1}_{#2}(#3)}

\newcommand{\hh}[1]{\operatorname{H}(#1)}
\newcommand{\HH}[2]{\operatorname{H}_{#1}(#2)}

\newcommand{\Ker}{\operatorname{Ker}}
\newcommand{\Coker}{\operatorname{Coker}}

\newcommand{\lra}{\longrightarrow}

\newcommand{\ov}{\overline}
\newcommand{\pd}{\operatorname{pd}}
\newcommand{\rank}{\operatorname{rank}}

\newcommand{\Ext}[4]{\operatorname{Ext}^{#1}_{#2}(#3,#4){}}
\newcommand{\Hom}[3]{\operatorname{Hom}_{#1}(#2,#3)}
\newcommand{\Tor}[4]{\operatorname{Tor}_{#1}^{#2}(#3,#4){}}
\newcommand{\topp}[3]{t_{#1}^{#2}(#3){}}

\newcommand{\ul}{\underline}
\newcommand{\ve}{{\varepsilon}}

\newcommand{\xra}{\xrightarrow}

\newcommand{\BC}{{\mathbb C}}
\newcommand{\BN}{{\mathbb N}}

\newcommand{\BZ}{{\mathbb Z}}

\newcommand{\bs}{\boldsymbol}

\newcommand{\fm}{{\mathfrak m}}
\newcommand{\fg}{\mathfrak{g}}
\newcommand{\fgl}{\mathfrak{gl}}
\newcommand{\fsl}{\mathfrak{sl}}
\newcommand{\fso}{\mathfrak{so}}
\newcommand{\fsp}{\mathfrak{sp}}

\newtheorem{theorem}[subsection]{Theorem}

\newtheorem{proposition}[subsection]{Proposition}
\newtheorem{lemma}[subsection]{Lemma}

\theoremstyle{definition}

\newtheorem*{ack}{Acknowledgements}

\theoremstyle{remark}
\newtheorem{remark}[subsection]{Remark}
\newtheorem*{Claim}{Claim}

\numberwithin{equation}{subsection}

\begin{document}

\title[Koszul properties of moment maps] {Koszul properties of  the moment map of\\some classical representations}

\author[A.~Conca]{Aldo~Conca}
\address{Dipartimento di Matematica, 
Universit\'a di Genova, Via Dodecaneso 35, 
I-16146 Genova, Italy}
\email{conca@dima.unige.it}

\author[H.-C.~Herbig]{Hans-Christian Herbig}
\address{Departamento de Matem\'atica Aplicada, Universidade Federal do Rio de Janeiro (UFRJ), Av. Athos da Silveira Ramos 149, Centro de
Tecnologia - Bloco C, CEP: 21941-909 - Rio de Janeiro, Brazil}
\email{herbig@labma.ufrj.br}

\author[S.~B.~Iyengar]{Srikanth B.~Iyengar}
\address{Department of Mathematics, University of Utah, Salt Lake City, UT 84112, U.S.A.}
\email{iyengar@math.utah.edu}
\date{\today}



\date{\today}
 
\begin{abstract}
This work concerns the moment map $\mu$ associated with the standard representation of a classical Lie algebra. For applications to deformation quantization it is desirable that $S/(\mu)$, the coordinate algebra of the zero fibre of $\mu$, be Koszul. The main result is that  this algebra is not Koszul for the standard representation of $\mathfrak{sl}_{n}$, and of $\mathfrak{sp}_{n}$. This is deduced from a computation of the Betti numbers of $S/(\mu)$ as an $S$-module, which are of interest also from the point of view of commutative algebra.
\end{abstract}

\maketitle

\section{introduction}
\label{sec:intro}

If one squeezes a nice and regular object into a mold that is too small, the object will crack and fold  and exhibit irregularities. This is precisely what happens when one is  representing a reductive complex Lie algebra $\fg$ on a low dimensional complex vector space $V$.
The irregularities become apparent when studying the \emph{moment map} $\mu\colon V\times V^*\to \fg^*$ of the representation; see Section \ref{se:moment-maps}.  When $\dim(V)$ is small the zero fibre $\mu^{-1}(0)$ cannot have codimension $\dim(\fg)$, which translates to the statement that the moment map cannot be a complete intersection in $V\times V^*$. 

When $V$ is spacious enough to properly host $\fg$, the moment map has a lot of good features. Tools to measure the size of the representation have been developed by  Schwarz \cite{LiftingDOs, DOsonQuotients} in connection with the \emph{modularity}, also known as modality~\cite{Vinberg}. Specifically, the moment map is a complete intersection if and only if the representation is \emph{$0$-modular}. A stronger condition is \emph{$1$-largeness}; then $\mu^{-1}(0)$ is irreducible and the ideal $(\mu)$ generated in $S:=\BC[V\times V^*]$ by the components of $\mu$ is radical. If the representation is even \emph{$2$-large}, then $\mu^{-1}(0)$ is also a normal variety.
When $\fg$ is simple, all but finitely many representations are 2-large; for details, see~\cite{KoszulComplex}.

This work reported in this paper concerns the Koszul property for the algebra $S/(\mu)$ and the following question:

\medskip

\emph{For which representations $\fg:V$ is the algebra $S/(\mu)$ Koszul?}

\medskip

The motivation to study this question comes from the Batalin-Fradkin-Vilkovisky approach to symplectic reduction~\cite{HT}. It has been observed by the second author that the BFV-quantization scheme of \cite{BHP} can be generalized to the situation when  $S/(\mu)$ is a Koszul algebra. The main reason for this is that in the Tate model~\cite{Tate} of $S/(\mu)$ over $S$, the internal degree is linked to homological degree; see \cite{ACI}. In the non-Koszul case however it appears that the BFV-quantization is unfeasible as the degrees proliferate which makes it is impossible to control anomalous terms.

When $\fg:V$  is $0$-modular, the ideal $(\mu)$ is a complete intersection of quadrics and then it is well known that  $S/(\mu)$ is a Koszul algebra; see, for example,~\cite{QuadraticAlgebras}.  So the question is moot only for representations that are not $0$-modular. For any representation $\fg:V$, the ideal $(\mu)$ is homogeneous quadratic, which is a necessary for $S/(\mu)$ to be Koszul. However, this condition is far from sufficient, and detecting when an algebra is Koszul is a notoriously difficult problem.

In this paper we develop methods to test when the algebra $S/(\mu)$ is \emph{not}  Koszul. We focus on the smallest (that is to say, the most `squeezed') nontrivial representations: the standard representations of the classical Lie algebras $\fgl_n,\fsl_n, \fso_n, \fsp_{n}$. Here is summary of our results.

\begin{theorem}
\label{thm:main}
Let $\fg:V$ denote the standard representation of the Lie algebras considered above, and let $R=S/(\mu)$ be the coordinate algebra of the zero fibre. The following statements hold for each $n\ge 1$.
\begin{enumerate}[\quad\rm(1)]
\item
For $\fg=\fgl_n$ the algebra $R$ is Koszul, $\dim R=n$ and $\depth R=1$.
\item
For $\fg=\fso_n$ the algebra $R$ is Koszul and $\dim R = n+1 = \depth R$. 
\item
For $\fg=\fsl_n$ the algebra $R$ is not Koszul, $\dim R=n$ and $\depth R=0$.
\item
For $\fg=\fsp_n$ the algebra $R$ is not Koszul, $\dim R=2n$ and $\depth R=0$.
\end{enumerate}
\end{theorem}

For $\fgl_n:\BC^{n}$ the ideal $(\mu)$ is generated by quadratic monomials, and this implies the Koszul property. For $\fso_n:\BC^{n}$ for $n\ge 3$ the components of the moment map (also known as angular momentum) are the $2\times 2$ minors of a generic $2\times n$ matrix and it is well-known that the corresponding algebra $S/(\mu)$ is Koszul. These assertions, and the claims about the dimension and depth, are justified in Section~ \ref{se:gl}.

The proofs of the assertion that $S/(\mu)$ is \emph{not} Koszul for $\fsl_n:\BC^n$ and $\fsp_{n}:\BC^{2n}$ for each $n\ge 2$ take up the bulk of this paper and require computing the Betti numbers of the corresponding algebra $S/(\mu)$, viewed as a module over the polynomial ring $S$; see Sections~\ref{se:sl} and \ref{se:sp}. Another crucial input in the proofs are certain obstructions to the Koszul property in terms of Betti tables, described in Section~\ref{se:koszul}. 

The Betti tables of these algebras are also of interest from a purely commutative algebra point of view, for they exhibit intriguing combinatorial patterns; this is discussed in Section~\ref{se:betti} of the Appendix.

There is another family of small representations for which $S/(\mu)$ is known not to be Koszul: the adjoint representations $\fgl_n:\fgl_n$ for $n\ge 3$. Here the zero locus of the moment map is essentially the \emph{commuting variety}, that is to say, the variety of pairs of commuting matrices in $\fgl_n\times \fgl_n$. The first few syzygies have been calculated by F. Hreinsdottir \cite{Hreinsdottir} and it follows from Remark \ref{re:notkoszul} that the corresponding algebra is not Koszul. A similar consideration also rules out $\fso_4:2k^4$ and $\fsl_3:2k^3$, for their Betti tables can be computed using, for example, Macaulay 2.

\section{Moment maps}
\label{se:moment-maps}

Let $\fg$ be a finite dimensional reductive Lie algebra over the field $k$ of characteristic zero and $\fg\to \fgl(V)$, $\xi\mapsto  (v\mapsto \xi.v)$ a finite dimensional representation of $\fg$; we write $\fg: V$ to indicate this situation. The \emph{moment map} associated to this representation is the bilinear form
\begin{alignat*}{2}
&\mu \colon V\times V^*\lra \fg^*& \quad &\text{given by}\\
&\mu (v,\alpha)(\xi) := \langle \alpha, \xi.v\rangle& \quad &\text{for $v\in V$, $\alpha\in V^*$ and $\xi\in \fg$.}
\end{alignat*}
Here $\langle \:, \:\rangle$ denotes the dual pairing between $V^*$ and $V$. Its dual is the map
\begin{align*}
&\fg\lra V\times V^* \quad  \text{with $\xi\mapsto \mu_\xi$, where}\\
&\mu_{\xi}(v,\alpha):=\langle \alpha, \xi.v\rangle\,.
\end{align*}
The map $\xi\mapsto \mu_\xi$  extends to a Lie algebra homomorphism $\fg\to S:=k[V\times V^*]$ when $S$ is equipped with the unique Poisson bracket $\{\:,\:\}$ such that 
\[
\{v,\alpha\}:=\langle \alpha,v\rangle\quad\text{and}\quad \{v,v'\}:=0=:\{\alpha,\alpha'\}
\]
for all $v,v'\in V$ and $\alpha,\alpha'\in V^*$. We are interested in the ideal of $S$ generated by the components of the moment map:
\[
 I:=(\mu)=\{\mu_\xi\mid \xi \in\fg\}\,.
 \]
This is a subalgebra with respect to the Poisson bracket $\{\:,\:\}$, but typically not an ideal with respect to the Poisson bracket. Since the $\mu_{\xi}$ are quadratic forms, $I$ is homogenous, when $S$ is viewed as a graded ring with $S_{1}=V\times V^{*}$; in fact, there is a natural $\BN^2$-grading on $S$, and $I$ is homogenous also with respect to this finer grading; this will prove to be useful in what follows.

Let us explain the coordinates expressions used for $V^*\times V$, $\fg$, $S$ and $\mu$.

\subsection{Standard representation of $\fgl_{n}$}
\label{mm:gl}
The Lie algebra $\fgl_{n}$ is the space of $n\times n$ matrices over $k$, with bracket $[A,B]=AB-BA$. Let $\fgl_n:k^n$ be the representation given by the obvious action of $\fgl_{n}$ on $k^{n}$. 

With respect to the canonical bases $e_1,e_2,\dots,e_n$ for $V=k^n$ and its dual basis $f_1,f_2,\dots f_n$ a point in $V\times V^*$ is identified with its linear coordinates 
\[
(p_1,p_2,\dots,p_n,q_1,q_2,\dots,q_n)=\sum q_i e_i+ p_i f_i\,.
\]
Accordingly we identify $k[V\times V^*]$ with
\[
S:=k
\bmatrix 
p_{1} & \dots & p_{n} \\
q_{1} & \dots & q_{n} 
\endbmatrix .
\]
For $i,j\in\{1,2,\dots n\}$, let $E_{ij}$ denote the $n\times n$-matrix with $1$ in the $i$th row and $j$th column, and zero elsewhere. These matrices form a basis for $\fgl_n$, and it is not hard to verify that the components of the moment map are $p_i q_j$ for $i,j=1,2,\dots,n$.

\subsection{Standard representation of $\fsl_{n}$}
\label{mm:sl}
The Lie algebra $\fsl_{n}$ is the Lie subalgebra of $\fgl_{n}$ consisting of matrices of trace zero. The representation $\fgl_{n}:k^{n}$ from \ref{mm:gl} thus restricts to one of $\fsl_{n}$. A basis for $\fsl_{n}$ is given by $E_{i,j}$ where $1\le i,j\le n$ with $i\ne j$,  and $E_{i,i}-E_{i+1,i+1}$ for $i=1,2\dots, n-1$. Hence the components of the moment map are $p_i q_j$ for ($1\le i,j\le n$ with $i\ne j$) and $p_i q_i-p_{i+1}q_{i+1}$ for $(i=1,2\dots, n-1)$.

\subsection{Standard representation of $\fso_{n}$}
\label{mm:so}
Recall that $\fso_{n}$ is the Lie subalgebra of $\fgl_{n}$ consisting of anti-symmetric matrices. The matrices $E_{i,j}-E_{j,i}$  where $1\le i<j\le n$ are a basis for $\fso_{n}$. Restricting the standard representation of $\fgl_{n}$ to $\fso_{n}$ yields a moment map with components $p_iq_j-p_j q_i$ $(1\le i<j\le n)$, that is to say, the $2\times 2$-minors of the generic   $2\times n$-matrix.

\subsection{Standard representation of $\fsp_{n}$}
\label{mm:sp}
The Lie algebra $\fsp_{n}$ consists of $2n\times 2n$-matrices $A$ such that
\[
A^{t}S=SA \quad \text{where $S=\begin{bmatrix} 0 & I_{n} \\ -I_{n} & 0\end{bmatrix}$}
\]
Here $I_{n}$ is the $n\times n$ identity matrix. A basis for $\fsp_n$ is given by the $2n\times 2n$ matrices
\begin{alignat*}{2}
&E_{i,j}-E_{n+j,n+i} &  & \text{ for } i,j=1,\dots, n \\
&E_{i,n+i} \text{ and } E_{n+i,i}&           &\text{ for } i=1,\dots, n \\
&E_{i,n+j}+E_{j,n+i} \text{ and } E_{n+i,j}+E_{n+j,i}& \qquad  &\text{ for }1\le i<j\le n
\end{alignat*}

We use $(p_{11},\dots ,p_{1n},p_{21},\dots ,p_{2n},q_{11},\dots ,q_{1n},q_{21},\dots ,q_{2n})$ for linear coordinates for $V^*\times V$ and identify $k[V\times V^*]$ with
\[
S:=k
\bmatrix 
p_{11} & \dots & p_{1n} & p_{21} & \dots &p_{2n} \\
q_{11} & \dots & q_{1n} & q_{21} & \dots &q_{2n} 
\endbmatrix
\]
The corresponding components of the moment map are
\begin{alignat*}{2}
&p_{1i}q_{1j}-p_{2j}q_{1i}  &  & \text{ for } i,j=1,\dots, n \\ 
&p_{1i}q_{2i} \text{ and } p_{2i}q_{1i}&      &\text{ for } i=1,\dots, n \\  
&p_{1i}q_{2j}+p_{1j}q_{2i} \text{ and }  p_{2i}q_{1j}+p_{2j}q_{1i} & \qquad &\text{ for }1\le i<j\le n
\end{alignat*}

\section{Koszul algebras}
\label{se:koszul}
In this section we recall the definition of certain invariants of modules over graded rings. The focus will be on Koszul algebras and some obstructions to this property.

\subsection{Fine grading}
The algebras and modules that we study will be equipped with a $\BN^{2}$-grading and the arguments exploit this structure. We prepare for this by introducing  notation concerning $\BN^{c}$-graded objects, for a positive integer $c\ge 1$.

Let $X$ be an $\BN^{c}$-graded set. When we speak of an element $x$ in $X$, it should be understood that $x$ is homogeneous: $x\in X_{v}$ for some $v\in \BN^{c}$. Then $v$ is the \emph{degree} of $x$, that we denote $\deg x$. For any $v$ in $\BN^{c}$ set
\[
\|v\|:=\sum_{i=1}^{c}v_{i}\quad\text{where $v=(v_{1},\dotsc,v_{c})$.}
\]
This is the \emph{total degree} of $v$. When $k$ is a field $X$ is a $\BN^{c}$-graded $k$-vector space, set
\[
\hilb X{s_{1},\dots,s_{c}}:= \sum_{v\in \BN^{c}}\rank_{k}(X_{v})s_{1}^{v_{1}}\cdots s_{c}^{v_{c}}\,,
\]
the \emph{Hilbert series} of $X$, viewed as an formal power series over $\BZ$ in indeterminates $s_{1},\dots,s_{c}$. We will also need the version based on total degrees
\[
\hilb X{s}:= \sum_{j\in \BN}\rank_{k} (X_{j})s^{j} = \hilb X{s,\dots,s} \,.
\]

\subsection{Graded rings}
Let $k$ be a field (of arbitrary characteristic) and $R$ a standard $\BN^{c}$-graded $k$-algebra, meaning:
\begin{enumerate}
\item
$R^{0}=k$;
\item
$R$ is an $\BN^{c}$-graded commutative $k$-algebra;
\item
$\rank_{k}(\oplus_{\|v\|=1} R_{v})$ is finite
\item
$R$ is generated by elements in $R_{v}$ with $\|v\|=1$.
\end{enumerate}
These conditions imply $R$ is noetherian and that $\rank_{k}R_{v}$ is finite for each $v\in \BN^{c}$. This ring is also local, in the graded sense: The set $\{r\in R\mid \|\deg(r)\|\ge 1\}$ is the unique homogeneous maximal ideal of $R$.

Any finitely generated $\BN^{c}$-graded $R$-module $M$ is noetherian, and hence $\rank_{k}M_{v}$ is finite for each $v\in \BN^{c}$.
The \emph{socle} of $M$ is the $k$-vector subspace
\[
\{m\in M \mid r\cdot m =0 \text{ for all $r\in R$ with $\|r\|\ge 1$}\}.
\]

Let $M,N$ be $\BN^{c}$-graded $R$-modules. Each $\Tor iRMN$ is  endowed with a $\BN^{c}$-grading compatible with the action of $R$. Thus $\Tor{}RMN$ is $(\BN\times\BN^{c})$-graded. We write $|{\,}|$ for the homological grading.  For example, for any $a$ in $\Tor{}RMN$
\[
|a| = i \quad\text{and}\quad \deg(a) = v\text{ means } a \in \Tor iRMN_{v}
\]
where $v\in\BN^{c}$. Then $\|\deg(a)\|$ is the total (internal) degree of $a$. 

\subsection{Betti numbers}
\label{def:betti}
Let $M$ be a finitely generated $\BN^{c}$-graded $R$-module. The  \emph{graded Betti numbers} of $M$ are the integers
\[
\betti {i,v}RM:= \rank_{k}\Tor iRkM_{v}\qquad \text{for $i\in \BN$ and  $v\in \BN^{c}$}
\]
and the associated generating function
\[
\po [R]M{s_{1},\dots,s_{c},u} := \sum_{i,v}\betti {i,v}RM s_{1}^{v_{1}}\cdots s_{c}^{v_{c}} u^{i}\,,
\]
viewed as an element in $\BZ[s_{1},\dots,s_{c}][|u|]$, is the \emph{graded Poincar\'e series} of $M$. Often, one is interested only in homological degrees and the associated Poincar\'e series:
\[
\betti iRM:= \sum_{v\in\BN^{c}} \betti {i,v}RM \qquad\text{and} \qquad \po [R]M{u}:= \sum_{i,j}\betti iRMu^{i}
\]
In what follows, we often compute the Poincar\'e series of $M$ as the Hilbert series of $\Ext {}RMk$. The natural grading on this $k$-vector space is cohomological and so, when dealing with them, we will tacitly switch to this one, by setting
\[
\Ext iRMk^{v}:=\Ext iRMk_{-v}\,,
\]
for any finitely generated $\BN^{c}$-graded $R$-module $M$.

We will also need to consider the following integers:
\[
\topp iRM:= \sup\{\|v\| \in\BN \mid \Tor iRkM_{v}\ne 0\} \qquad \text{for each $i\in \BN$.}
\]
This is the highest total degree of a Betti number of $M$ in homological degree $i$.

\subsection{Koszul algebras}
\label{def:koszul}
When $R$ is regular (that is to say, isomorphic to the symmetric $k$-algebra on $R_{1}$) one has $\topp iRk=i$ for  $0\le i\le \dim R$.  Even when $R$ is singular (meaning, not regular) it is not hard to see that 
\[
\topp iRk\ge i\qquad \text{for $i\ge 0$.}
\]
When equality holds for each $i$ one says that the $k$-algebra $R$ is \emph{Koszul}.  In this case, it is not hard to verify that there is an equality of formal power series
\[
\po [R]k{-u}\hilb R{u} = 1\,.
\]
It is known that the converse also holds; see, for example,~\cite[Theorem 1]{Fr}.

When $R$ is Koszul, the relations defining $R$ must be quadratic; in other words, there is an isomorphism of $k$-algebra $R\cong S/I$, where $S$ is a standard graded polynomial ring over $k$ and $I$ can be generated by quadratic forms. The converse holds if $I$ is monomial---see, for example, \cite[Theorem~15]{CDR}---but not in general.

\begin{remark}
\label{re:topR-k}
If for some integer $s\ge 0$ one has $\topp iSR\le i+1$ for each $i\le s$, then 
\[
\topp iRk\le i\quad\text{for $i\le s+1$.}
\]
Moreover, if $s=\pd_{S}R$, the projective dimension of $R$ over $S$, then $R$ is Koszul.

These results are immediate from the graded version of a coefficientwise inequality \cite[Proposition~3.3.2]{Av:barca} due to Serre, relating the Poincar\'e series of $k$ over $R$, and of $k$ and $R$ over $S$.
\end{remark}

The next result is sometimes helpful is detecting when a $k$-algebra is \emph{not} Koszul.

\begin{remark}
\label{re:notkoszul}
If $R$ is Koszul, then $\beta^{S}_{1}(R)=\beta^{S}_{1}(R)_{2}$ and 
\[
\beta^{S}_i(R)_{2i}\le \binom{\beta^{S}_{1}(R)}i \quad \text{for each integer $i\ge 1$.}
\]
This inequality is due to Avramov, Conca, and Iyengar~\cite[Theorem~5.1]{ACI2}.
\end{remark}

A different obstruction to the Koszul property involves Massey products on the Koszul homology of $R$. 

\subsection{Matric Massey products}
Let $R$ be a standard $\BN^{c}$-graded $k$-algebra. Let $K$ be the Koszul complex of $R$ and set $H:=\hh K$.  The canonical DG (= Differential Graded) $R$-algebra structure on $K$ is  compatible with the $\BN^{c}$-grading. More precisely, for elements $a,b$ in $K$ one has
\begin{alignat*}{2}
&|ab| = |a| + |b|\,& \quad &\text{and}\quad \deg(ab) = \deg(a) + \deg(b) \\
&|d(a)| = |a|-1\,& \quad &\text{and}\quad \deg(d(a)) = \deg(a)\,.
\end{alignat*}

We recall the notion of Massey products on $H$, for which Kraines~\cite{Kr} is a good reference, keeping track of the finer grading involved. To begin with, for any element $a$ in $K$, one sets
\[
\ov a :=   (-1)^{|a|+1}a\,.
\]
Let $w_{1},\dots,w_{s}$ be elements in $H_{\ges 1}$; as always, these are assumed to be homogenous with respect to the $(\BN\times\BN^{c})$-grading on $H$.  The Massey product $\langle w_{1},\dots,w_{s}\rangle$ is said to be \emph{defined} if there is a collection of elements $a_{ij}$ in $K$, for $1\le i < j\le s$ and $(i,j)\ne (1,s)$ such that, for each $i$ the element $a_{ii}$ is cycle representing $w_{i}$ and for each pair $(i,j)$ as above, one has
\begin{equation}
\label{eq:mm-input}
d(a_{ij}) = \sum_{r=i}^{j-1} \ov{a}_{ir}a_{r+1,j}\,.
\end{equation}
Such a collection $\{a_{ij}\}$ of elements is a \emph{defining system} for  $\langle w_{1},\dots,w_{s}\rangle$. Given a defining system, it is readily verified that the element 
\begin{equation}
\label{eq:mm-output}
a_{1s}:=\sum_{r=1}^{s-1}\ov{a}_{1r}a_{r+1,s}
\end{equation}
is a cycle in $K$. The Massey product $\langle w_{1},\dots,w_{s}\rangle$, when it is defined, is the collection of all elements in $H$ that can be represented by cycles of the form $a_{1s}$ for some defining system $\{a_{ij}\}$ as above.  Observe that these lie in $H_{p,v}$ where 
\[
p=\sum_{r=1}^{s}|w_{i}| + (s-2)\quad\text{and}\quad v=\sum_{r=1}^{s}\deg(w_{i})\,.
\]

We need an extension of all this to matrices with coefficients in $H$. These \emph{matric Massey products} where introduced by May~\cite{May} in the context of general DG algebras and brought to bear to study commutative rings (specifically, Koszul homology) by Avramov~\cite{Av}. In this context, one considers matrices $W_{1},\dots,W_{s}$ with entries in $H$ with appropriate restrictions on their sizes and the degrees (both homological and internal) of the entries in them to ensure that the products involved in \eqref{eq:mm-input} and \eqref{eq:mm-output} are permissible; each $a_{ij}$ is replaced by a matrix $A_{ij}$ with entries in $K$. 

Elements of the matric Massey product $\langle W_{1},\dots, W_{s}\rangle$, when it is defined, are matrices with entries in $H_{\ges 1}$.  
For any matrix $W$ in  $\langle W_{1},\dots, W_{s}\rangle$, the degrees of $W_{i,j}$, the entry in position $(i,j)$ of $W$, are given by 
  \begin{align*}
|W_{ij}| &= \sum_{r=1}^{s}|(W_{r})_{i_{r-1}i_{r}}| + (s-2) \\
\deg(W_{ij}) &=  \sum_{r=1}^{s} \deg((W_{r})_{i_{r-1}i_{r}})
\end{align*}
where $i_{0}=i$ and $i_{r}=s$, and $i_{1},\dots,i_{s-1}$ are any permissible choice of integers; for example, $1,\dots,1$.

An element $w$ in $H$ is \emph{decomposable} as a matric Massey product if it belongs to $\langle W_{1},\cdots,W_{s}\rangle $ for 
some  system of matrices with entries in $H_{\ges 1}$; necessarily, $W_{1}$ is then a row matrix and $W_{s}$ is a column matrix. 

\begin{lemma}
\label{le:massey}
Let $R$ be a standard $\BN^{c}$-graded $k$-algebra and $H$ its Koszul homology. If $s$ is an integer such that $\topp iRk\le i$ for each $i\le s$, then the elements of $H_{i,v}$ are decomposable as matric Massey products when $2\le i \le s-1$ and $\|\deg v\|\ge i+2$.
\end{lemma}

\begin{proof}
In what follows, we  set
\[
T:=\Tor{}Rkk/(\Tor 1Rkk)\,.
\]
This inherits from $\Tor{}Rkk$ a structure of a ($\BN\times\BN^{c}$)-graded $k$-algebra, and our hypothesis translates to $T_{i,v}=0$ when $i\le s$ and $\|\deg v\|\ge i+1$. The crux of the argument is that there is a $k$-linear map $\sigma \colon H_{\ges 1} \lra T$ of homological degree one, called the suspension; see, for example, \cite[Definition 3.7]{GM}. This map respects the internal gradings, that is to say, for each $i\ge 1$ and $v\in\BN$ one has
\[
\sigma_{i,v}\colon H_{i,v}\lra T_{i+1,v}\,.
\]
Thus, the hypothesis on then $\topp iRk$ entails $T_{i+1,v}=0$ whenever $1\le i\le s-1$ and $\|\deg v\|\ge i+2$. Therefore for these $i,v$ one has $H_{i,v}\subseteq \Ker(\sigma)$. It remains to note that $\Ker(\sigma)$ consists precisely of elements decomposable as matric Massey products in $H_{\ges 1}$; see \cite[Corollary 5.13]{GM}.
\end{proof}

\section{The general linear and the special orthogonal case}
\label{se:gl}
In this section we describe the Betti numbers of the moment maps associated to canonical representation of the general linear, and the special orthogonal, Lie algebra.  The moment maps in question are described in \ref{mm:gl} and \ref{mm:so}. These results are special cases of general results in commutative algebra. The main reason for describing the computations carefully is that they get used in the later sections.

\subsection{The general linear group}
\label{ss:gln}
Let $k$ be a field and $S$ the polynomial ring over $k$ in indeterminates $\ul p:=p_{1},\dots,p_{n}$ and $\ul q :=q_{1},\dots,q_{n}$. Thus
\[
S:=k
\bmatrix 
p_{1} & \dots & p_{n} \\
q_{1} & \dots & q_{n} 
\endbmatrix
\]
We view $S$ as an $\BN^{2}$-graded $k$-algebra with each $p_{i}$ of degree $[1,0]$ and each $q_{i}$ of degree $[0,1]$.  Set
\[
I: = (\ul p )\cap (\ul q) \quad\text{and}\quad A = S/I\,.
\]
Since $I$ is generated by quadratic monomials, the $k$-algebra $A$ is Koszul; see Definition~\ref{def:koszul}. Next we compute the Betti numbers of $A$ as an $S$-module.

As $I$ is  homogenous, the $k$-algebra $A$  inherits the  $\BN^{2}$-grading of $S$.  Noting that $(\ul p) + (\ul q)$ is the maximal ideal of $S$, the standard Mayer-Vietoris sequence for the ideals $(\ul p)$ and $(\ul q)$ reads
\begin{equation}
\label{eq:mvs}
0\lra A \xra{\ \iota\ }C \xra{\ \pi\ } k \lra 0\,, \quad\text{where $C= S/(\ul p) \bigoplus S/(\ul q)$.}
\end{equation}
From this it follows that the bigraded Hilbert series of $A$ is
\begin{equation}
\label{eq:hs-A}
\hilb{A}{s,t} = \frac 1{(1-s)^{n}} + \frac 1{(1-t)^{n}} - 1 
\end{equation}

Next we compute the Poincar\'e series of $A$, viewed as an $S$-module. 

\subsection{Structure of $\Ext{}SAk$ as module over $\Ext{}Skk$}
Set $\fm = (\ul p,\ul q)$; this is the homogenous maximal ideal of $S$. The residue classes modulo $\fm^{2}$ of the sequence $\ul p,\ul q$ is a basis for the $k$-vector space $\fm/\fm^{2}$. Let $e_{1},\dots,e_{n},f_{1},\dots,f_{n}$ be the dual basis, of the $k$-vector space
\[
V:=\Hom k{\fm/\fm^{2}}k\,.
\]
We view this as a $\BN^{2}$-graded $k$-vector space with (upper grading) $\deg(e_{i})=[1,0]$ and $\deg(f_{i})=[0,1]$, in cohomological degree one. Then one has
\[
\Lambda:= \Ext{}Skk \cong \wedge V\quad\text{and}\quad \Ext{}SCk\cong \Lambda/(\ul f) \oplus \Lambda/(\ul e)\,,
\]
and the map $\Ext{}S{\pi}k$, with $\pi$ as in \eqref{eq:mvs}, is  the canonical map
\[
\Lambda \lra \Lambda/(\ul f) \oplus \Lambda/(\ul e)\,.
\]
It is easy to check that this map is surjective in homological degrees $\ge 1$; for example,  in degree one the map is an isomorphism and, as $\Lambda$-modules, the source and target are generated by their components in homological degree zero. It follows that  \eqref{eq:mvs} induces an exact sequence of graded $k$-vector spaces 
\begin{equation}
\label{eq:mvs-ext}
0\lra  \Ext{\ges 1}SAk \lra \Ext{\ges 2}{S}kk [0,0,1] \lra \Ext{\ges 2}SCk [0,0,1] \lra 0\,.
\end{equation}
Noting that $\Ext{}Skk$ is an exterior algebra on $2n$ generators, half of which have (upper) degree $(1,0,1)$ and  the rest have degree $(0,1,1)$. Thus one gets
\begin{equation}
\label{eq:mvs-po}
\po k{s,t,u}:= (1+su)^{n}(1+tu)^{n} \quad\text{and}\quad \po C{s,t,u}:= (1+su)^{n} + (1+tu)^{n}
\end{equation}
This discussion leads to the following result. 

\begin{theorem}
\label{th:gl}
Let $k$ be a field and let $S$ and $A$ be the $k$-algebras introduced in \ref{ss:gln}.
The trigraded Poincar\'e series of $A$ over $S$ is 
\[
\po A{s,t,u}:= 1 + u^{-1}((1+su)^{n}-1)((1+tu)^{n} - 1)\,.
\]
The nonzero graded Betti numbers of $A$ over $S$ are 
\[
\beta^{S}_{i,v}(A) = \begin{cases} 
1 & \text{for $i=0$ and $v=[0,0]$} \\
\binom n{v_{1}} \binom n{v_{2}}  \quad & \text{for $i\ge 1$, $v_{i}\ge 1$, and $v_{1}+v_{2}=i+1$}  
\end{cases}
\]
The Poincar\'e series is $1+u^{-1}((1+u)^{n}-1)^{2}$. The Betti numbers of $A$ over $S$ are
\[
\beta^{S}_{i}(A) = \begin{cases} 
1 & \text{for $i=0$} \\
\binom {2n}{i+1} - 2\binom{n}{i+1} \quad &\text{for $i\ge 1$}\,. 
\end{cases} 
\]
Moreover, $\dim A = n$ and $\depth A=1$. 
\end{theorem}

\begin{proof}
The claims about the Poincar\'e series are immediate from \eqref{eq:mvs-ext} and \eqref{eq:mvs-po}. The ones about the dimension and the depth of $A$ are  from \eqref{eq:mvs}.
\end{proof}

For later use, we record that following result that is immediate from \eqref{eq:mvs-ext}.

\begin{lemma}
\label{le:extSAk}
\pushQED{\qed}
As a $\Lambda$-module, $\Ext{\ges 1}SAk[0,0,-1]$ is isomorphic to the ideal 
\[
(e_{j}\wedge f_{j}\mid 1\leq i,j\leq n)\ \text{of $\Lambda$}\qedhere
\]
\end{lemma}

This completes our discussion on the general linear Lie algebra.

\subsection{Special orthogonal Lie algebra}
Let $k$ be a field and $S$ the polynomial ring over $k$ in indeterminates $\ul p:=p_{1},\dots,p_{n}$ and $\ul q :=q_{1},\dots,q_{n}$. Thus
\[
S:=k
\bmatrix 
p_{1} & \dots & p_{n} \\
q_{1} & \dots & q_{n} 
\endbmatrix
\]
As before, $S$ is to be viewed as an $\BN^{2}$-graded $k$-algebra with each $p_{i}$ of degree $[1,0]$ and each $q_{i}$ of degree $[0,1]$.  Let $I$ be the $2\times 2$ minors of the matrix above and set $A = S/I$.   The bigraded Hilbert series of $A$ is 
\[
 \frac{1- st\sum_{i=0}^{n-2}  (-1)^i  \binom n{2+i}  h_i(s,t )}{(1-s)^{n}(1-t)^{n}}
\]
where $h_i(s,t)$ is the sum  of all monomials of degree $i$ in $s$ and $t$. In particular
\[
\hilb As = \frac{1+s(n-1)}{(1-s)^{n+1}}\,.
\]
The minimal free resolution of $A$ over $S$ is given by the Eagon-Northcott complex; see, for example, \cite[Section~2C]{BV}. It yields $\topp iSA\le i+1$ for each $i$, and hence $A$ is Koszul; see Remark~\ref{re:topR-k} From this one gets that the Poincar\'e series of $k$ over $A$ is
\[
\po [A]ku = \frac{(1+u)^{n+1}}{1 - u(n-1)}
\]

\section{The special linear case}
\label{se:sl}
In this section we study the moment map associated to the canonical representation of the special linear Lie algebra, $\fsl_{n}$; see \ref{mm:sl}.

Let $k$ be a field; for what follows, its characteristic has to be zero, or at least larger than $(n+1)/2$, for we need Lemma~\ref{le:relation}. As in Section~\ref{se:gl}, let $S$ be the polynomial ring over $k$ in indeterminates $\ul p:=p_{1},\dots,p_{n}$ in degree $[1,0]$, and $\ul q :=q_{1},\dots,q_{n}$, in degree $[0,1]$. Set
\[
B = S/I \quad \text{where $I =(\{p_{i}q_{j},p_{i}q_{i}-p_{j}q_{j}\mid 1\le i,j\le n\text{ with } i\ne j, \})$.}
\]
The ideal $I$ is homogenous, with respect to the $\BN^{2}$ grading on $S$, and hence the $k$-algebra $B$ inherits this grading. 

It is straightforward to check that $p_{1}q_{1}$ is in the socle of $B$ and that $B/(p_{1}q_{1})=A$, the $k$-algebra encountered in Section~\ref{se:gl}. There is thus an exact sequence of $\BN^{2}$-graded $S$-modules
\begin{equation}
\label{eq:abseq}
0\lra k[-1,-1] \xra{\ 1\mapsto p_{1}q_{1}\ } B \lra A \lra 0
\end{equation}
and from this one and \eqref{eq:hs-A} one obtains that the bigraded Hilbert series of $B$:
\[
\hilb{B}{s,t} = \frac 1{(1-s)^{n}} + \frac 1{(1-t)^{n}} - 1 + st \,.
\]

Next we determine the Poincar\'e series of $B$, as an $S$-module. To this end we compute $\Ext{}SBk$, exploiting the fact that this is a $\BN^{3}$-graded module over the $\BN^{3}$-graded $k$-algebra
\[
\Lambda = \Ext{}{S}kk\,.
\]
By \eqref{eq:mvs-po} one gets
\[
\hilb{\Lambda}{s,t,u} = (1+su)^{n}(1+tu)^{n}
\]
where $u$ is the cohomological degree. The exact sequence \eqref{eq:abseq} induces an exact sequence of graded $\Lambda$-modules
\[
\lra \Lambda [-1,-1,-1] \xra{\ \beta\ } \Ext{}SAk \lra \Ext{}SBk \lra \Lambda [-1,-1,0]\xra{\ \beta[1]\ }
\]
This gives an equality of formal power series
\[
\po B{s,t,u} = \hilb{\Coker(\beta)}{s,t,u} + u^{-1}\cdot \hilb{\Ker(\beta)}{s,t,u}\,.
\]
It thus remains to compute the Hilbert series of $\Ker(\beta)$ and $\Coker(\beta)$.

Let $C$ be the $S$-module $S/(\ul p)\oplus S/(\ul q)$, encountered in Section~\ref{se:gl}; see \eqref{eq:mvs}. The exact sequence \eqref{eq:mvs-ext}, reproduced in the top row of the diagram below, is one of graded $\Lambda$-modules:
\[
\xymatrixcolsep{1.3pc}
\xymatrixrowsep{1.3pc}
\xymatrix{
0\ar@{->}[r] & \Ext{\ges 1}SAk \ar@{->}[r]^-{\ \gamma\ } & \Lambda [0,0,1] \ar@{->}[r] 	&\Ext{}SCk [0,0,1] \ar@{->}[r] & 0  \\
 & \Lambda[-1,-1,-1] \ar@{->}[u]^{\beta}\ar@{=}[r] & \Lambda[-1,-1,-1]\ar@{->}[u]^{\gamma\beta}}
\]
The Snake Lemma then gives $\Ker(\beta) = \Ker (\gamma\beta)$ and an exact sequence 
\[
0\lra \Coker(\beta) \lra \Coker(\gamma\beta) \lra \Ext{}SCk [0,0,1]\lra 0\,,
\]
of graded $\Lambda$-modules. It thus remains to study the map
\[
\Lambda[-1,-1,-1] \xra{\ \gamma\beta\ } \Lambda [0,0,1]
\]
This map is $\Lambda$-linear, so it is determined by $\gamma\beta(1)$, where $1\in\Lambda$.
To identify this element, as in Section~\ref{se:gl}, let $e_{1},\dots,e_{n},f_{1},\dots,f_{n}$ be the basis of 
\[
\Lambda^{1} = \Hom k{\fm/\fm^{2}}k
\]
dual to that of the images of $\ul p,\ul q$ in $\fm/\fm^{2}$; here $\fm=(\ul p,\ul q)$.

\begin{Claim}
$\gamma\beta(1) = \sum_{1\les i\les n}e_{i}f_{i}$ in $\Lambda^{2}$.
\end{Claim}

Indeed, $\beta(1)$ is the element in $\Ext 1SAk$ corresponding to the exact sequence \eqref{eq:abseq}, and $\gamma$ is the map defined by
concatenation with the sequence \eqref{eq:mvs}. It follows that $\gamma\beta(1)$ is the element of $\Lambda^{2}$ corresponding to the exact sequence
\begin{equation}
\label{eq:2ext}
0\lra k[-1,-1] \xra{\ 1\mapsto p_{1}q_{1}\ } B \xra{\ \iota\ } C  \xra{\ \pi\ } k \lra 0
\end{equation}
To identify this element in $\Lambda$, let $E$ be the Koszul complex of $\ul p,\ul q$. We view it as the exterior algebra $\bigwedge 
\bigoplus_{i} (SP_{i}\oplus SQ_{i})$ with $P_{i}\mapsto p_{i}$ and $Q_{i}\mapsto q_{i}$. The first steps of a lifting of the natural morphism $E\to k$ to the complex above is given below:
\[
\xymatrixcolsep{2pc}
\xymatrixrowsep{1pc}
\xymatrix{
0\ar@{->}[r] &k[-1,-1] \ar@{->}[r]^{\ 1\mapsto p_{1}q_{1}\ } & B  \ar@{->}[r]^-{\iota} & C 
					\ar@{->}[r]^-{\pi} &k \ar@{->}[r] & 0 \\
\cdots\ar@{->}[r] &{\bigoplus}_{i,j} S P_{i}Q_{j} \ar@{->}[u]^{\kappa_{2}} \ar@{->}[r] & {\bigoplus}_{i} (S P_{i}\oplus SQ_{i}) \ar@{->}[u]^{\kappa_{1}} 					\ar@{->}[r] & S \ar@{->}[u]^{\kappa_{0}} \ar@{->}[r] &k \ar@{=}[u] \ar@{->}[r] & 0}
\]
where the maps $\kappa_{i}$ are given by $\kappa_{0}(1) = (0,1)$, and								
\begin{gather*}
\kappa_{1}(P_{i}) = p_{i} \quad \text{and}\quad \kappa_{1}(Q_{i}) = 0 \quad \text{for each $i$} \\
\kappa_{2}(P_{i}P_{j}) = 0\,, \quad \kappa_{2}(Q_{i}Q_{j}) = 0\,, \quad \text{and}
\quad \kappa_{2}(P_{i}Q_{j}) = 1 \quad \text{for each $i,j$}.
\end{gather*}
The claim follows as \eqref{eq:2ext} corresponds to the class of $\kappa_{2}$ in $\Lambda^{2}=\HH{-2}{\Hom{S}Ek}$.

\medskip

From the preceding claim and Proposition~\ref{pr:maximal-rank}, it follows that the map $\gamma\beta$ has maximal rank, and hence that
\begin{gather*}
\hilb{\Ker{\beta}}{s,t,u} = u [(stu^{2}-1)(1+su)^{n}(1+tu)^{n}]_{+} \\
\hilb{\Coker{\beta}}{s,t,u} = u [(1-stu^{2})(1+su)^{n}(1+tu)^{n}]_{+} - u ((1+su)^{n} + (1+tu)^{n})
\end{gather*}
where $[h(s,t,u)]_{+}$ denotes positive part of $h(s,t,u)$; namely, those terms whose coefficients are negative integers are omitted.

\medskip

Summing up one gets the result below.

\begin{theorem}
\label{th:sl}
\pushQED{\qed}
Let $k$ be a field, and let $S$ and $B$ the $k$-algebras introduced at the beginning of this section; in particular,  the characteristic of $k$ is zero, or greater than $(n+1)/2$.  The trigraded Poincar\'e series of $B$ over $S$ is 
\begin{align*}
2 - (1+su)^{n}-(1+tu)^{n} 
	& \, +  \, [(1-stu^{2})(1+su)^{n}(1+tu)^{n}]_{+} \\
	& \, +  \,   u^{-2}[(stu^{2}-1)(1+su)^{n}(1+tu)^{n}]_{+} 
\end{align*}
The Poincar\'e series is thus
\[
1 - 2[(1+u)^{n}-1] + u^{-1}[(1-u^{2})(1+u)^{2n}-1]_{+}\, + \,    u^{-2}[(u^{2}-1)(1+u)^{2n}]_{+} 
\]
So the Betti numbers of $B$ over $S$ are
\[
\beta^{S}_{i}(B) = \begin{cases} 
1 & \text{for $i=0$} \\
\binom {2n}{i+1} - \binom{2n}{i-1}- 2 \binom n{i+1} & \text{for $1\leq i\leq n-1$} \\
\binom {2n}{i} - \binom{2n}{i+2} & \text{for $n\leq i\leq 2n$}
\end{cases}
\]
Moreover, $\dim B = n$ and $\depth B=0$.  \qedhere
\end{theorem}

Next we focus on the (lack of) the Koszul property of $B$.

\begin{proposition}
\label{pr:sl-not-koszul}
One has $\topp iBk=i$ for $0\le i\le n$ and $\topp {n+1}Bk=n+2$; in particular, the $k$-algebra $B$ is not Koszul.
\end{proposition}

\begin{proof}
From Theorem~\ref{th:sl} yields $\topp iSB = i+1$ for $i\le n-1$, and hence $\topp iBk=i$ for $0\le i\le n$, by Remark \ref{re:topR-k}.
It remains to verify the assertion about $\topp{n+1}Bk$.

Let $H$ denote the Koszul homology algebra of $B$. This is a $(\BN\times \BN^{2})$-graded algebra, with
\[
H_{i,v} \cong \Tor iSBk_{v} \quad \text{for $i\in \BN$ and $v\in \BN^{2}$.}
\]
Thus $\hilb{H}{s,t,u}$ is the Poincar\'e series of $B$ over $S$,  computed in Theorem~\ref{th:sl}. It follows from that result if $v_{1}=0$ or $v_{2}=0$, then $H_{i},v=0$. Thus, for any element $h\in H_{i,v}$ that is decomposable as a matric Massey product, one must have $v_{1}\ge 2$ and $v_{2}\ge 2$. On the other hand, one has
\[
H_{n+1,v} \cong k \quad \text{for $v=[n+1,1]$ and $v=[1,n+1]$.}
\]
Thus, elements in these subspaces of $H$ are indecomposable as matric Massey products. It now remains to  apply Lemma~\ref{le:massey} to deduce that $\topp {n+1}Bk>n+1$.
\end{proof}

\section{The symplectic case}
\label{se:sp}
In this section we study the moment map associated to the canonical representation of the symplectic Lie algebra, $\fsp_{n}$; see \ref{mm:sp}. Let $k$ be a field of characteristic not equal to two and $S$ the polynomial ring in $4n$ indeterminates:
\[
S:=k
\bmatrix 
p_{11} & \dots & p_{1n} & p_{21} & \dots &p_{2n} \\
q_{11} & \dots & q_{1n} & q_{21} & \dots &q_{2n} 
\endbmatrix
\]
Set $\ul p = \{p_{ij}\}_{i,j}$ and $\ul q = \{q_{ij}\}_{i,j}$. As before, we view $S$ as an $\BN^{2}$-graded $k$-algebra with the $\ul p$ of degree $[1,0]$ and the $\ul q$ of degree $[0,1]$.  Let $I$ be the ideal generated by the $2n^{2}+ n$ homogenous quadratic forms
\begin{gather}
 p_{1i}q_{1j}-p_{2i}q_{2j}  \quad 1\leq i,j\leq n  \\
 p_{1i}q_{2i}, p_{2i}q_{1i} \quad 1\leq i\leq n \\
 p_{1i}q_{2j} + p_{1j}q_{2i}  \quad 1\leq i<j\leq n  \\
 p_{2i}q_{1j} + p_{2j}q_{1i}  \quad 1\leq i<j\leq n  
\end{gather}
The focus of this section is on the $k$-algebra $Z:=S/I$, which is also $\BN^{2}$-graded. 

\begin{lemma}
\label{le:sp-socle}
The socle of $Z$ is $J:= (\ul p)(\ul q)$, viewed as an ideal in $Z$. It is a vector space of rank $2n^2 - n$, concentrated in degree $[1,1]$.
\end{lemma}

\begin{proof}
The main task is to verify that $J$ is in the socle of $Z$; then, since the ring $Z/J$ is  reduced (it is isomorphic to $S/(\ul p)(\ul q)$) it would follow that $J$ is equal to the socle of $Z$. Moreover, then the assertion about the rank of $J$ is clear.  

For the proof, it is convenient to relabel the variables $\ul p$ and $\ul q$ as follows:
\[
a_{i}:=p_{1i}\,,\quad b_{i}:=q_{1i}\,,\quad c_{i}:=p_{2i}\,,\quad d_{i}:=q_{2i}\quad\text{for each $i$.}
\]
Since the characteristic of $k$ is not two, it is easy to verify that the ideal $I$ can be generated by the elements 
\begin{enumerate}[{\quad\rm(1)}]
\item $a_{i}b_{j} - c_{i}d_{j}$
\item $a_{i}d_{j} + d_{i}a_{j}$
\item $c_{i}b_{j} + b_{i}c_{j}$
\end{enumerate}
where $1\le i,j\le n$. Recall that $Z$ is $\BN^{2}$-graded, with each $a_{i}$ and $c_{i}$ have degree $[1,0]$ while each $b_{i}$ and $d_{i}$ has degree $[0,1]$. The desired result is that the component $Z_{[1,1]}$ is the socle of $Z$. This is equivalent to the assertion that $Z_{[2,1]}=0=Z_{[1,2]}$. By symmetry it suffices to prove that $Z_{[2,1]}=0$, that is to say, that the monomials of the type
\[
aab, \quad aad, \quad  acb, \quad  acd, \quad  ccb, \quad  ccd
\]	
are all zero in $Z$. By virtue of equation (1) above, monomials of the type $aab$, respectively $acb$, are equal to those of the type $acd$, respectively $ccd$. It thus suffices to consider the following types of monomials: $aad$, $acd$, $ccb$, and  $ccd$.

\medskip
\emph{Types} $aad$ and $ccb$: Using equation (2), one gets that
\[
a_{i}a_{j}d_{k}= -a_{i}d_{j}a_{k} = -a_{k}a_{i}d_{j}\,.
\]
Note that the indices on the right can be obtained from the left by applying the three cycle $(1\,2\,3)$, and this involves a change of sign.
So, applying the same rule three time one gets that
\[
a_{i}a_{j}d_{k}=  - a_{i}a_{j}d_{k} 
\] 
and hence that $a_{i}a_{j}d_{k}=0$; recall that the characteristic of $k$ is not two.

A similar argument, now using (3), shows that $c_{i}c_{j}b_{k}=0$ for all $i,j,k$.

\medskip

\emph{Type} $acd$: Now using equation (1) twice and then equation (2), one gets
\[
a_{i}c_{j}d_{k} = a_{i}a_{j}b_{k} = c_{i}a_{j}d_{k} = - c_{i}d_{j}a_{k} = -a_{k}c_{i}d_{j}
\]
Once again, the indices on the right are obtained from those on the left by applying a three cycle, along with a change of sign. So applying this thrice we deduce that $a_{i}c_{j}d_{k} = 0$, as desired.

\medskip

\emph{Type} $ccd$: Applying (1), then (3), and then again (1) yields
\[
c_{i}c_{j}d_{k} = c_{i}a_{j}b_{k}= - b_{i}a_{j}c_{k} = -d_{i}c_{j}c_{k} = -c_{j}c_{k}d_{i}
\]	
so arguing as before one deduces that $c_{i}c_{j}d_{k}=0$. 
\end{proof}

Set $A=S/(\ul p)(\ul q)$. By Lemma~\ref{le:sp-socle} one has an exact sequence 
\begin{equation}
\label{eq:azseq}
0\lra J \lra Z \xra{\ \ve\ } A\lra 0
\end{equation}
of $\BN^{2}$-graded $S$-modules. Thus the bigraded Hilbert series of $Z$ is
\[
\hilb{Z}{s,t} = \frac 1{(1-s)^{2n}} + \frac 1{(1-t)^{2n}} - 1 + (2n^{2}-n)st \,.
\]
We compute the Poincar\'e series of $Z$. The crucial computation is the following.

\begin{lemma}
\label{le:sp-map}
The exact sequence \eqref{eq:azseq}  induces the following isomorphisms and exact sequences of $\BN^{2}$-graded $k$-vector spaces:
\begin{gather*}
\Ext{0}SAk\xra{\ \cong\ } \Ext{0}SZk \\
0\lra \Hom kJk \lra \Ext{1}SAk\lra \Ext{1}SZk\lra 0 \\
\Ext{1}SJk\xra{\ \cong\ } \Ext{2}SAk \\
0\lra \Ext{i}SZk \lra \Ext{i}SJk\lra \Ext{i+1}SAk\lra 0 
\end{gather*}
for each $i\ge 2$.
\end{lemma}

\begin{proof}
Indeed, the first isomorphism is clear. Given this, applying $\Hom S-k$ to the exact sequence \eqref{eq:azseq} yields an exact sequence 
\[
0\lra \Hom kJk \lra \Ext{1}SAk\lra \Ext{1}SZk\xra{\ \eta\ } \Ext{1}SJk\xra{\ \eth\ }\Ext{2}SAk 
\]
of $\BN^{2}$-graded $k$-vector spaces. We claim $\eta=0$ and $\eth$ is an isomorphism, which justifies all but the last assertion in the statement.  To verify the claim, note that
\begin{align*}
\rank_{k}\eta 
	& = \rank_{k} \Ext{1}SZk - \rank_{k} \Ext{1}SAk + \rank_{k} \Hom kJk \\
	&= (2n^{2}+n) - 4n^{2} + (2n^{2}-n) \\
	&= 0 
\end{align*}
where the second equality comes from Lemma~\ref{le:sp-socle}. Thus $\eta=0$ and hence the map $\eth$ is injective.  As $\Ext 1SJk$ is isomorphic to $\Ext 1Skk\otimes_{k}\Hom kJk$, one gets the first equality below
\begin{align*}
\rank_{k} \Ext 1SJk & = (2n^{2}-n) \rank_{k}\Ext 1Skk = (2n^{2}-n)4n \\
\rank_{k} \Ext 2SAk & = \binom {4n}3 - 2\binom {2n}3 = (2n^{2}-n)4n 
\end{align*}
The second equality is from Theorem~\ref{th:gl}; keep in mind that the $\ul p$ and $\ul q$ have been doubled. Thus it follows that $\eth$ is an isomorphism, as desired.

To verify the exactness of the last family of sequences, it suffices to prove that the map 
\[
\Ext{\ges 1}SJk \lra \Ext{\ges 2}SAk
\]
induced by \eqref{eq:azseq} is surjective. This map is linear with respect to the action of $\Lambda:=\Ext{}Skk$. Since $\Ext{\ges 2}SAk$ is generated by $\Ext 2SAk$ as a $\Lambda$-module, by Lemma~\ref{le:extSAk}, it  suffices to verify that the map of $k$-vector spaces
\[
\Ext 1SJk \lra \Ext 2SAk
\]
 is surjective. This is the map $\eth$ that we already know is an isomorphism. 
\end{proof}

Lemma~\ref{le:sp-map} implies that the minimal resolution of $Z$ over $S$ is pure, and in fact linear after the second syzygy.
From the previous claim and Theorem~\ref{th:gl} (for $2n$) one gets the following result.

\begin{theorem}
\label{th:sp}
Let $k$ be a field, and let $S$ and $A$ the $k$-algebras introduced at the beginning of this section; in particular,  the characteristic of $k$ is zero, or odd. The nonzero graded Betti numbers of $A$ over $S$ are 
\[
\beta^{S}_{i,v}(A) = \begin{cases} 
1   & \text{$i=0$ and $v=[0,0]$} \\
(2n^{2}+n)   & \text{$i=1$ and  $v=[1,1]$}\\
(2n^{2}-n)\binom {2n}{v_{1}-1} \binom {2n}{v_{2}-1} - \binom {2n}{v_{1}} \binom {2n}{v_{2}}
		   & \text{$i\ge 2$, $v_{i}\ge 1$, and $v_{1}+v_{2}=i+2$} 
\end{cases}
\]
The Betti numbers of $Z$ over $S$ are
\[
\beta^{S}_{i}(A) = \begin{cases} 
1 & \text{for $i=0$} \\
2n^{2}+n & \text{for $i=1$} \\
(2n^{2}-n)\binom {4n}{i} - \binom{4n}{i+2}+ 2\binom{2n}{i+2}\quad &\text{for $i\ge 2$} 
\end{cases} 
\]
Moreover, $\dim Z = 2n$ and $\depth Z=0$. 
\end{theorem}

\begin{proof}
The assertions about the (graded) Betti numbers have already been justified. The dimension and the depth of $Z$ are from \eqref{eq:azseq}.
\end{proof}

This has the following corollary. 

\begin{proposition}
\label{pr:sp-koszul}
The $k$-algebra $Z$ is not Koszul.
\end{proposition}

\begin{proof}
From Theorem~\ref{th:sp} one gets the first equality below
\[
\beta^{S}_2(Z)_{4} = \beta^{S}_2(Z) = \frac 53 n^{2}(4n^{2}-1) >  \frac 12 (2n^{2}+n)(2n^{2}+n-1)=\binom{\beta^{S}_{1}(Z)}2
\]
The inequality is readily verified. Now apply Remark~\ref{re:notkoszul}. 
\end{proof}

\appendix

\section{Quadratic exterior forms and maximal rank}
The main task of this section is to prove Proposition~\ref{pr:maximal-rank}, which was used in the proof of Theorem~\ref{th:sl}.
Throughout $k$ will be a field. For elements $\bs x=x_1,\dots, x_u$ in a commutative ring, we write $s_i(\bs x)$ for the $i$-th elementary symmetric polynomials in $\bs x$  with the convention that $s_0(\bs x)=1$ and $s_i(\bs x)=0$ for $i<0$ or $i>u$. 

\begin{lemma} 
\label{le:relation} 
In  $k[x_1,\dots, x_u, y_1,\dots, y_v]/J$, where $J=(y_1^2,\dots, y_v^2,x_1^2,\dots, x_u^2)$, for every $d\geq 0$ there is an equality 
\[
(s_1(\bs x)+s_1(\bs y))\sum_{k=0}^d (-1)^k  k! (d-k)! s_{k}(\bs x) s_{d-k}(\bs y)=(d+1)! (s_{d+1}(\bs y)+(-1)^d  s_{d+1}(\bs x))\,.
\]
\end{lemma}
 
\begin{proof} 
\pushQED{\qed}
It is not hard to verify that there is an equality 
\[
s_1(\bs x)s_{k}(\bs x)=(k+1)s_{k+1}(\bs x)\,; 
\]
by symmetry, this implies also that $s_1(\bs y)s_{d-k}(\bs y)=(d-k+1)s_{d-k+1}(\bs y)$. Using these one gets equalities
\begin{align*}
 (s_1(\bs x)+s_1(\bs y))\sum_{k=0}^d (-1)^k &  k! (d-k)! s_{k}(\bs x)s_{d-k}(\bs y) \\
 & = \sum_{k=0}^d (-1)^k   k! (d+1-k)! s_{k}(\bs x)s_{d-k+1}(\bs y)  \\
   &\phantom{this is not a sum} +\sum_{k=0}^d (-1)^k  (k+1)! (d-k)! s_{k+1}(\bs x) s_{d-k}(\bs y)\\
&=(d+1)! (s_{d+1}(\bs y)+(-1)^d  s_{d+1}(\bs x)) \qedhere
\end{align*} 
\end{proof}

Let $V$ be a $k$-vector space of dimension $2n$, and $\Lambda$ the exterior algebra on $V$. Let $e_1,\dots, e_n, f_1,\dots, f_n$ be a basis of $V$ and set 
\[
w:=\sum_{i=1}^n e_i f_i\,;
\]
this is an element in $\Lambda_{2}$. The result below is well known when $\chr k=0$ and can be deduced as a special case of the Hard Lefschetz Theorem; see~\cite[page~122]{GH}. However, we have been unable to find an argument that covers also the case of positive characteristics in the literature.

\begin{proposition}
\label{pr:maximal-rank}
If $\chr k=0$ or $\chr k>(n+1)/2$, then the multiplication map $w\colon \Lambda_i\to \Lambda_{i+2}$ has maximal rank for every $i$; in other words, $w\colon \Lambda_i\to \Lambda_{i+2}$ is injective for $i\leq n-1$ and surjective for $i\geq n-1$.  
\end{proposition}

\begin{proof}
To begin with, there is an isomorphism $\Hom k{\Lambda}k\cong \Lambda (2n)$ of $\Lambda$-modules and from this it follows that it suffices to verify that multiplication by $w$ is injective in degrees $\le n-1$. Moreover, it suffices to verify injectivity for $i=n-1$.

Indeed, assume $w\colon \Lambda_{i}\to \Lambda_{i+2}$ is injective for some $i\le n-1$. For any nonzero element $\lambda$ of degree $i-1$, there exists an element $\nu$ in $\Lambda_{1}$ for which $\lambda\nu\ne 0$; this is because the socle of $\Lambda$ is $\Lambda_{2n}$. Then $w\lambda\nu \ne 0$ implies $w\lambda \ne 0$. An iteration yields the desired result.

To reiterate: It suffices to verify that $w\colon \Lambda_{n-1}\to \Lambda_{n+1}$ is injective. Since the ranks of the source and target coincide, this is equivalent to verifying that the map is surjective.  We prove this by an induction on $n$, the base case  $n=1$ being obvious. 

Suppose $n\ge 2$. Then $\Lambda_{n+1}$ has a $k$-basis consisting of monomials
\[
\mu =e_{i_1} \dots  e_{i_a}  f_{j_1} \dots  f_{j_b}
\]
where $1\leq a,b\leq n$ with $a+b=n+1$, and $i_1<\dots<i_a$ and $j_1<\dots j_b$,  

If for some $h$ we have $i_h\not\in \{j_1,\dots, j_b\}$, then  $e_{i_h}w=e_{i_h}w_1$ with 
\[
w_1=\sum_{v=1, v\neq i_h}^n e_v f_v
\]
and, by induction on $n$, there exists $\lambda \in \Lambda_{n-2}(U)$ such that 
\[
\lambda w_1=e_{i_1} \dots \hat e_{i_h}  \dots e_{i_a}  f_{j_1} \dots   f_{j_b}\,.
\]
Here $U$ is the $k$-subspace of $V$ generated by  $\{e_1,\dots, e_n, f_1,\dots, f_n\} \setminus \{e_{i_h}, f_{i_h}\}$. Then 
\[
\nu e_{i_h} w= \nu e_{i_h} w_1=\pm e_{i_h}  \nu w_1=\pm \mu\,.
\]
A similar argument settles the case when there exists $h$ with $j_h\not\in \{i_1,\dots, i_a\}$.

It remains to consider the case $a=b$ and  $\{i_1,\dots, i_a\}=\{j_1,\dots,j_b\}$.  Set $m=a=b$ so that $n=2m-1$. We can assume (renaming the indices) that $i_h=j_h=h$ for $h=1,\dots, m$, that is to say, $\mu=\prod_{i=1}^m e_i f_i$. 

We apply Lemma~\ref{le:relation}  as follows: we set $v=m$,  $u=m-1$ and $d=m-1$  and  
\begin{alignat*}{2}
&x_i\to e_{m+i} f_{m+i}&  \qquad &\text{ for }i=1,\dots, m-1\\
&y_i\to e_i f_i& \qquad &\text{  for } i=1,\dots, m \,.
\end{alignat*}
Such a specialization makes sense because the elements $x_{i}y_i$ have  square zero and  commute among themselves. 
Since  $s_1(s')+s_1(t')=w$,  $s_{m}(t')=\prod_{i=1}^m  s_i f_i$ and  $s_{m}(s')=0$ we obtain: 
\[
w\sum_{k=0}^{m-1} (-1)^k  k! (m-1-k)! s_{k}(s') s_{m-1-k}(t')=m! \prod_{i=1}^m  s_i f_i.
\]
Since we are assuming that $\chr k=0$ or $\chr k>(n+1)/2$ we have that $m!$ is invertible and hence we may conclude that  $w\nu =\prod_{i=1}^m  s_i f_i$ with 
\[
\nu =\frac{1}{(m!)} \sum_{k=0}^{m-1} (-1)^k  k! (m-1-k)! s_{k}(s') s_{m-1-k}(t').
\]
This completes the proof.
\end{proof}

\section{Betti Tables and numerology}
\label{se:betti}
In this section we collect some observation, and questions, concerning the algebras studied in this work. To begin with, consider the following Betti tables for the representations $\fsl_n:k^n$,  computed in Theorem~\ref{th:sl}, for the first few values of $n$.

\footnotesize{
\begin{eqnarray*}
\begin{array}{c|c|ccccccccccccc}
&&0&1&2&3&4&5&6&7&8&9&10&11&12\\\hline
\fsl_2:k^2&0&1&-&-&-&-&&&&&&&&\\
&1&-&3&-&-&-&&&&&&&&\\
&2&-&-&\fbox{5}&4&1&&&&&&&&\\\hline
\fsl_3:k^3&0&1&-&- &-&-&-&-&&&&&&\\
&1&-&8&12&-&-&-&-&&&&&&\\
&2&-&-& -&\fbox{14}&14&6&1&&&&&&\\\hline
\fsl_4:k^4&0&1&- &- & -&- &- & -&-&-&&&&\\
&1&-&15&40&40&- &- & -&-&-&&&&\\
&2&-&- & -& -&\fbox{42}&48&27&8&1&&&&\\\hline
\fsl_5:k^5&0&1&- &- & -  &-   &-   &   -&  -&- & -&-&&\\
&1&-&24&90&155 &130 &-   &   -&  -&- & -&-&&\\
&2&-&- & -& -&-&\fbox{132}&165&110&44&10&1&&\\\hline
\fsl_6:k^6&0&1&- &-  & - &-   &-   &         -&  -& - &  -& -& -& -\\
&1&-&35&168&399&560 &427 &         -&  -& - &  -& -& -& -\\
&2&-&- &  -&  -&-   &  - &\fbox{429}&572&429&208&65&12& 1
\end{array}
\end{eqnarray*}
}

\normalsize

The framed  numbers in the table are the \emph{Catalan numbers}; see \cite[A000108]{OEIS}. This is true for each $n$, and what is more, the \emph{Catalan triangle} introduced by Shapiro \cite{LouisShapiro}, appears as well. 

Indeed, the sequence $(C_n)_{n\ge 0}$ of Catalan numbers, which has numerous combinatorial interpretations~\cite{Stanley, Koshy}, can be defined as follows:
\[
C_n:={\frac{1}{n+1}}\binom{2n}n =\binom{2n}n- \binom{2n}{n+1}\,.
\]
The  \emph{Segner's recursion formula} $C_{n+1}=\sum_{i,j:i+j=n} C_i C_j$ can be rewritten in terms of the generating function $C(x)=\sum_{n\ge 0} C_n x^n$ as follows: $x\,C^2(x)=C(x)-C_0$.

The entries $B(N,r)$ of the Catalan triangle, with $N,r \ge 1$, are defined by the higher moments $\gamma^r(x)=:\sum_{N\ge 1}B(N,r)\:x^N$ of the generating function $\gamma(x)=C(x)-C_0=C(x)-1$. In other words, 
\begin{align*} B(N,r)=\sum_{\begin{subarray}{l}  
        i_1,i_2,\dots, i_r\ge 1\\
        \sum_j i_j=N    
      \end{subarray}} C_{i_1}C_{i_2}\cdots C_{i_r}.
\end{align*}
It has been shown in \cite{LouisShapiro} that $B(N,r)= \frac rN \binom{2N}{N-r}$, which also makes sense for $r<0$.

The Catalan triangle can be recovered when taking differences of columns in the Pascal triangle:
\begin{align*}
\binom{2n}i - \binom{2n}{i+2} & = \left(1- \frac{(2n-i)(2n-i-1)}{(i+1)(i+2)}\right) \binom{2n}i \\
&=2\:  \frac{(i+1-n)(2n+1)}{(i+1)(i+2)}\: \binom{2n} i\\
&= \frac{i+1-n}{n+1}\: \frac{(2n+1)(2n+2)}{(i+1)(i+2)}\: \binom{2n} i\\
&=\frac{i+1-n}{n+1}\: \binom{2n+2}{i+2}\\
&=B(n+1,i+1-n)
\end{align*}
This proves the occurrence of the Catalan triangle in the Betti table of $\fsl_n:k^n$; compare Theorem \ref{th:sl}. The many combinatorial interpretations of the Catalan triangle~\cite{Koshy,LouisShapiro} raise the question: \emph{Is there a minimal free resolution of the moment map of $\fsl_n:k^n$ that underlies the occurrence of the Catalan triangles}?

\medskip

The Betti tables of the moment map for the representations $\fsp_n:k^{2n}$, computed in Theorem \ref{th:sp}, are also worth pointing to.
For $n=1$ the Betti table is the same as that of $\fsl_2:k^{2}$. For $n=2,3$ one gets the following.

\footnotesize
{
\begin{eqnarray*}
\begin{array}{c|c|ccccccccccccc}
& &0&  1&      2&    3&   4&    5&  6&  7&    8&9&10&11&12\\\hline
\fsp_2: k^{4}&0&1&  -&      -&    -&   -&    -&  -&  -&    -&&&&\\
&1&-& 10&      -&    -&   -&    -&  -&  -&    -&&&&\\
&2&-&  -&    100&  280& 392&  328&167& 48&    6&&&&\\\hline
\fsp_3: k^{6}&0&1&  -&      -&    -&   -&     -&    -&    -&   -&   -&  -& - &  - \\
&1&-& 21&      -&    -&   -&     -&    -&    -&   -&   -&  -& - &  - \\
&2&-&  -&   525 & 2520& 6503&11088&13365&11660&7359&3288&989&180& 15
\end{array}
\end{eqnarray*}
}

\normalsize

The noteworthy feature here is that the Betti table has only two strands, and the jump from one to the next occurs after the second step.

\subsection*{Poincar\'e series}
Finally we discuss the Poincar\'e series of $k$ over the coordinate algebra of the zero fibre. Let $R$ be a standard graded algebra, as in Section~\ref{se:koszul}. For simplicity, we focus on the $\BN$-graded case. We are interested in the Betti numbers of $k$ as an $R$-module. A basic question is when the corresponding Poincar\'e series $\po [R]k{u}$ is rational; it is not so for a general graded ring $R$; see~\cite[\S4.3]{Av:barca}. \emph{What about for algebras of the form $S/(\mu)$, for a moment map $\mu$?}

The Poincar\'e series for a complete intersection is rational~\cite[Theorem~9.2.1]{Av:barca}, so this takes care of $0$-modular representations. Thus, as for the Koszul property, the question is moot only for small representations. Rationality also holds when $R$ is Koszul, for then there is an equality $\po[R]k{s,u}= \hilb R{-su}^{-1}$. This settles the case of the standard representations $\fgl_{n}:k^{n}$ and $\fso_{n}:k^{n}$.

We do not know the answer for $\fsl_{n}:k^{n}$ and $\fsp_{n}:k^{n}$. For $\fsl_2:k^2$ Gr\"obner basis calculations (using Macaulay2) suggest that the Betti table of the minimal free resolution of $k$ is upper triangular, with a new strand appearing after every $3$ steps in the homological degree.  The following formula for the Poincar\'e series best fits the available numerical data
\begin{align*}
\po [S/(\mu)]k{s,u} ={(1+us)^2\over (1-us)^3(1+us)-2u^3 s^4}\,.
\end{align*}
For $\fsl_3:k^3$ we also encounter a triangular shape but we have been unable to guess what the Poincar\'e series might be.

Following the arXiv posting of an earlier version of this manuscript, Jan-Erik Roos wrote to us (email, dated 17 June 2017) that he could answer some of the questions  posed in the preceding paragraphs. Pointing to his paper~\cite{Roos}, he writes ``I have in principle solved ``all" cases of the Homological Behaviour of families of quadratic forms in four variables...'' Roos notes that for $\fsl_2:k^2$  the ring $S/(\mu)$ is isomorphic to the one in \cite[Case 7, pp. 427]{Roos}, and that the corresponding Poincar\'e series is precisely the one we proposed above. Furthermore, he proposes the following  formula for the Poincar\'e series for the case of $\fsl_{3}:k^{3}$
\begin{align*}
\po [S/(\mu)]k{s,u} ={(1+us)^3\over (1-us)(1-2us-4u^2s^2-2u^3s^3+ u^4s^4)-2u^4 s^5}\,.
\end{align*}
and suggests a method to tackle $\fsl_n:k^n$, for arbitrary $n$; we hope to develop these ideas in due course.

\begin{ack}
Our thanks to Lucho Avramov for helpful conversations regarding this work; in particular, for pointing out Lemma~\ref{le:massey}, and the work of Hreinsdottir~\cite{Hreinsdottir}. Part of this article is based on work supported by the National Science Foundation under Grant No.\,0932078000, while AC and SBI were in residence at the Mathematical Sciences Research Institute in Berkeley, California, during the 2012--2013 Special Year in Commutative Algebra.   AC was supported by INdAM-GNSAGA and PRIN ``Geometry of Algebraic Varieties'' 2015EYPTSB\_{008}.  SBI was partly supported by NSF grants DMS-1503044.
\end{ack}

 \end{document}